\documentclass[12pt,english]{article}
\usepackage[T1]{fontenc}
\usepackage[latin9]{inputenc}
\usepackage[margin=0.8in]{geometry}
\usepackage{amsmath}
\usepackage{enumerate}
\usepackage{amsfonts}
\usepackage{amssymb}
\usepackage{amsthm}
\usepackage{esint}
\usepackage{babel}
\usepackage{cite}
\usepackage{graphicx}
\usepackage{authblk}
\usepackage{url}
\theoremstyle{plain}
\newtheorem*{theorem*}{Theorem}
\newtheorem{thm}{Theorem}[section]
\newtheorem{lem}[thm]{Lemma}
\newtheorem{prop}[thm]{Proposition}
\newtheorem{cor}[thm]{Corollary}
\theoremstyle{definition}
\newtheorem{defn}[thm]{Definition}
\newtheorem{exmp}[thm]{Example}
\newtheorem{rem}[thm]{Remark}
\numberwithin{equation}{section}

\title{Toeplitz operators on Bergman spaces of polygonal domains}
\date{\vspace{-4ex}}
\begin{document}
\author{Paula Mannersalo}
\maketitle
\begin{abstract}
We study the boundedness of Toeplitz operators with locally integrable symbols on Bergman spaces $A^p(\Omega),$ $1<p<\infty,$ where $\Omega\subset \mathbb{C}$ is a bounded simply connected domain with \hyphenation{poly-gonal} polygonal boundary. We give sufficient conditions for the boundedness of
generalized Toeplitz operators in terms of ''averages'' of symbol over certain Cartesian squares.
We use the Whitney decomposition of $\Omega$ in the proof. We also give examples of bounded Toeplitz operators on $A^p(\Omega)$ in the case where polygon $\Omega$ has such a large corner that the Bergman projection is unbounded.
\end{abstract}
\section{Introduction and notation}
We continue the study of generalized Toeplitz operators with locally integrable symbols on Bergman spaces $A^p(\Omega),$ $1<p<\infty$ (see \cite{mannersalo, taskinen1, taskinen2}). In \cite{taskinen1, taskinen2} it was considered the boundedness of \hyphenation{ge-neralized} generalized Toeplitz operators on Bergman spaces $A^p(\mathbb{D})$ of the unit disk. The result of \cite{taskinen1} was generalized in \cite{mannersalo} to the case of a bounded simply connected domain $\Omega\subset \mathbb{C}$ with $C^4$ smooth boundary.
In this paper we release the smoothness assumption on $\Omega$ and we consider polygonal domains of $\mathbb{C}$ with finite number of corners. Our main result, Theorem \ref{main1}, contains a weak sufficient condition for the boundedness of Toeplitz operators in such domains. The crucial difference compared to \cite{mannersalo} is that the derivative of the Riemann conformal mapping from $\Omega$ onto the unit disk is not always bounded. We also consider the apparently complicated question of finding bounded Toeplitz operators in the situation where the Bergman projection is unbounded.
We mainly follow the notation and terminology of \cite{mannersalo}, and the technical approach is similar to those of \cite{mannersalo, taskinen1, taskinen2}.


We start with the basic definitions and notation. 
By $C, C', C_1$ etc. we mean positive constants independent of given functions or indices, but which can vary from place to place.
Suppose $x$ and $y$ are two positive quantities. 
By writing $x \sim y,$ we mean that $x$ and $y$ are comparable i.e. that there exist two (absolute) constants $C_1,C_2>0$ such that $C_1x<y<C_2x.$
The unit disk in the complex plane is denoted by $\mathbb{D}$.
Let $\Omega\subset \mathbb{C}$ be a $n$-sided polygon, i.e., a bounded simply connected domain whose boundary consists of $n$ line segments and $n$ corners. More precisely, by a corner we mean inner angle $\alpha\pi,$ where $0<\alpha<2, \alpha\neq 1$. Let polygon $\Omega$ have inner angles $\alpha_1\pi,...,\alpha_n\pi$ at the corresponding vertices $w_1,...,w_n \in \partial\Omega$, where $n\geq 3$. We say that a corner is outward if $0<\alpha<1$ and inward if $1<\alpha<2$. We use the notation $\boldsymbol{\alpha_m}:=\text{max}_k(\alpha_k)$ for the factor of the maximum angle (corner), and the corresponding vertex is denoted by $\boldsymbol{w_m}\in \partial\Omega$.
Let $\varphi:\Omega\rightarrow \mathbb{D}$ be a Riemann conformal mapping. Its inverse is denoted by $\psi=\varphi^{-1}:\mathbb{D}\rightarrow\Omega.$  According to the Schwarz-Christoffel formula
\begin{equation}
\label{schwarz0}
\psi(z)=A\int_{0}^{z}\prod_{k=1}^{n}(1-\overline{z}_kz)^{\alpha_k-1}+B,\;\;\;\;z\in \mathbb{D},
\end{equation}
where $z_k\in \partial \mathbb{D}$, $\psi(z_k)=w_k$, $\Sigma_{k=1}^n\alpha_{k}=n-2$ and $A$ and $B$ are constants which determine the size and the position of $\Omega$. The formula \eqref{schwarz0} is not unique: when the vertices $w_1,...,w_n\in \partial \Omega$ are given, three of the points $z_k \in \partial\mathbb{D}$ can be chosen arbitrarily. For details about the Schwarz-Christoffel formula we refer to \cite[p.189--196]{nehari}.
It follows from \eqref{schwarz0} that
\begin{equation}
\label{schwarz1}
\psi'(z)=A\prod_{k=1}^{n}(1-\overline{z}_kz)^{\alpha_k-1},\;\;\;z\in \mathbb{D},\;\;A\neq 0,
\end{equation}
and
\begin{equation}
\label{schwarz2}
\varphi'(w)=A^{-1}\prod_{k=1}^{n}(1-\overline{\varphi(w_k)}\varphi(w))^{1-\alpha_k},\;\;\;w\in \Omega,
\end{equation}
where $w_k=\psi(z_k)\in \partial\Omega$ are the vertices. 
Euclidean disk with center $z$ and radius $r>0$ is denoted by $B(z,r)$.
The Bergman space $A^p(\Omega)$ $(1< p <\infty)$ consists of analytic functions $f:\Omega\rightarrow \mathbb{C}$ satisfying
\[
\left\Vert f\right\Vert_p:=\left(\int_{\Omega}|f(w)|^pdA(w)\right)^{1/p}<\infty,
\]
where $dA(w):=\pi^{-1}dxdy$ ($w=x+iy$) is the scaled area measure in the plane.
If there is a chance for confusion, we denote by $\left\Vert f\right\Vert_{p,\Omega}$ the norm on $\Omega$ and by $\left\Vert f\right\Vert_{p,\mathbb{D}}$ the norm on $\mathbb{D}.$ By $A_{\omega}^p(\Omega)$ we mean weighted Bergman space equipped with the norm $\left\Vert f\right\Vert_{p,\omega}=\int_{\Omega}|f|^p\omega dA$, where $\omega:\Omega\rightarrow \mathbb{R_{+}}$ is a positive real-valued weight. The index $p$ of the Bergman space $A^p(\Omega)$ is $1<p<\infty$ throughout the paper.\par
Let $P_{\Omega}$ be the Bergman projection, i.e., the orthogonal projection from $L^2(\Omega)$ onto $A^2(\Omega)$. It has the integral representation
$(P_{\Omega}f)(z)=\int_{\Omega}K_\Omega(z,w)f(w)dA(w),$
where 
\begin{equation}
\label{kernel}
K_\Omega(z,w)=\frac{\varphi'(z)\overline{\varphi'(w)}}{(1-\varphi(z)\overline{\varphi(w)})^2}
\end{equation} is the Bergman kernel of $\Omega$
(see \cite{duren, bekolle2, hedenmalm, shikhvatov}).
To simplify notation we denote $G_{\Omega}(z,w):=1-\varphi(z)\overline{\varphi(w)}$.
In particular, the Bergman projection of the unit disk, $P_{\mathbb{D}}:L^2(\mathbb{D}) \rightarrow A^2(\mathbb{D})$, can be expressed as
\begin{equation}
\label{P_unitdisk}
(P_{\mathbb{D}}f)(z)=\int_{\mathbb{D}}\frac{f(w)dA(w)}{(1-z\overline{w})^2}.
\end{equation}
The classical Toeplitz operator $T_a$ on the Bergman space $A^p(\Omega)$ is defined by
$T_{a}(f)=P_{\Omega}(af),$ if $a:\Omega\rightarrow \mathbb{C}$, the symbol of $T_a$, is such that the following integral converges:
\begin{equation}
\label{Tomega}
(T_{a}f)(z)=\int_{\Omega}K_{\Omega}(z,w)a(w)f(w)dA(w),\;\;\;f\in A^p(\Omega).
\end{equation}
In this article symbol $a:\Omega\rightarrow \mathbb{C}$ is always at least locally integrable, i.e., it belongs to $L_{loc}^1(\Omega).$
We need also the maximal Bergman projection $P^+_{\Omega}$ on $L^p(\Omega)$ defined as
\begin{equation}
\label{maximalbergman}
(P_{\Omega}^{+}f)(z):=\int_{\Omega}|K_{\Omega}(z,w)||f(w)|dA(w),
\end{equation}
if the integral converges.

On several occasions we deal with the boundary distance weight $v(w):=\text{dist}(w,\partial\Omega):=\inf_{z\in \partial \Omega}|w-z|$, $w\in \Omega.$ It follows from the Koebe distortion theorem \cite[Corollary 1.4]{pommerenke} that
\begin{equation}
\label{weight}
v(w)=\text{dist}(w,\partial\Omega)\sim \frac{1-|\varphi(w)|^2}{|\varphi'(w)|},\;\;w\in\Omega.
\end{equation}

For $u+iv\in \Omega,$ we denote Cartesian squares by
\[
S:=S(u+iv,\rho):=\{x+iy\in \Omega\,|\, u\leq x \leq u+\rho,\, v\leq y\leq v+\rho\},
\]
where $\rho>0$ is the side length of $S$. Here $\rho$ is always so small that $S(u+iv,\rho)\subset \Omega.$ The area of $S$ is denoted by $|S|:=\rho^2$, and
the diameter of $S$ is $\text{diam}S:=\sup_{w_1,w_2\in S}|w_1-w_2|=\sqrt{2}\rho.$ The distance of $S$ to the boundary $\partial \Omega$ is denoted by $\text{dist}(S,\partial\Omega):=\inf_{w\in S,z\in \partial \Omega}|w-z|.$
We now form a partition of $\Omega\subset \mathbb{C}$ using Whitney's decomposition:
There exist $z_{n}:=x_{n}+iy_{n}\in \Omega$ 
and $\rho_n>0$ for all $n\in\mathbb{Z_{+}}$ such that the squares 
\begin{equation}
	\label{Sn}
	S_{n}:=\{x+iy\,|\, x_{n}\leq x \leq x_{n}+\rho_{n},\, y_{n}\leq y \leq y_{n}+\rho_{n}\} \subset \Omega
\end{equation}
form a partition of the domain $\Omega$ and  
$\text{diam}S_n\sim \text{dist}(S_n,\partial\Omega)$ for all $n.$
In addition to $S_n$ we need a little bit larger squares $\tilde{S}_n\supset S_n(x_n+iy_n,\rho_n)$ with side lengths $\frac{11}{10}\rho_n,$
\begin{equation}
	\label{tildeS}
	\tilde{S}_{n}:=\{x+iy\,|\, x_{n}-\frac{1}{20}\rho_n\leq x \leq x_{n}+\frac{21}{20}\rho_{n},\, y_{n}-\frac{1}{20}\rho_n\leq y \leq y_{n}+\frac{21}{20}\rho_{n}\}.
\end{equation}
More specifically, the Whitney decomposition has the following properties (see \cite{stein}). 
\begin{lem}\emph{\cite{stein}}
	\label{whitney}
	Let $\Omega\subset \mathbb{C}$ be an open non-empty set. 
	There exist squares \newline $S_n:=S_n(x_n+iy_n,\rho_n)\subset\Omega$ $(n=1,2,3...)$ such that
	\begin{enumerate}[$(a)$]
		\item $\Omega=\bigcup\limits_{n}S_n$,
		\item $S_n$ are mutually disjoint,
		\item 
		$\mathrm{diam}S_n \leq \mathrm{dist}(S_n,\partial\Omega)  \leq 4  \,
		\mathrm{diam}S_n
		$ and since $\mathrm{diam}S_n=\sqrt{2}\rho_n$,	
		\[\label{rho}
		\sqrt{2}\rho_n\leq \mathrm{dist}(S_n,\partial\Omega)  \leq 4 \sqrt{2}\rho_n,\]
		\item each point of $\Omega$ is contained in at most $144$ of the squares $\tilde{S}_n.$
	\end{enumerate}
\end{lem}
In the definition of the generalized Toeplitz operator we use the partial sum operator $T_{a,\Omega}^{(m)}$:
\begin{equation}
\label{Tm}
T_{a,\Omega}^{(m)}f(z):=\sum_{n=1}^{m}\int_{S_n}K_{\Omega}(z,w)a(w)f(w)dA(w)
\end{equation} for all integers $m\geq 1$.
In order to avoid confusion, we always denote by $T_{a}$ the usual Toeplitz operator
\eqref{Tomega} and by $T_{a,\Omega}$ the generalized Toeplitz operator defined as follows.
\begin{defn}
	Let $\{S_n\}_n$ be a Whitney decomposition of $\Omega$ and let $a\in L_{loc}^1(\Omega).$ 
	Given $f\in A^p(\Omega)$,
	we define the generalized Toeplitz operator
	\begin{equation}
	\label{generalizedT}
	(T_{a,\Omega}f)(z):=\lim_{m\rightarrow\infty}T_{a,\Omega}^{(m)}f(z),\;\;\;z\in\Omega,
	\end{equation}
	if the limit exists for all $z\in \Omega$.
\end{defn}
\begin{rem}
	\label{coincide}
 If $f\in A^p(\Omega)$ is such that $\varphi'af\in L^1(\Omega)$, then $T_af=T_{a,\Omega}f$, i.e., the usual \eqref{Tomega} and the generalized \eqref{generalizedT} definitions of the Toeplitz operator coincide. Indeed, if $z\in \Omega$ and $\mbox{\large$\chi$}_{S_n}$ is the characteristic function of $S_n$, then
\begin{equation}
\nonumber
|(T_af)(z)|\leq\int_{\Omega}|K_{\Omega}(z,w)||a(w)||f(w)|dA(w)\leq C_z\int_{\Omega}|\overline{\varphi'(w)}||a(w)||f(w)|dA(w)< \infty.
\end{equation}
Therefore, due to the Lebesgue's dominated  convergence theorem, we have
\begin{eqnarray}
\nonumber
	(T_af)(z)&=&\int_{\Omega}K_{\Omega}(z,w)a(w)f(w)dA(w)=\lim_{m\rightarrow \infty}\int_{\Omega}\big(\sum_{n=1}^m\mbox{\large$\chi$}_{S_n}(w)\big)K_{\Omega}(z,w)a(w)f(w)dA(w)
	\\\nonumber &=&\lim_{m\rightarrow \infty}\sum_{n=1}^m\int_{\Omega}\mbox{\large$\chi$}_{S_n}(w)K_{\Omega}(z,w)a(w)f(w)dA(w)=(T_{a,\Omega}f)(z).
\end{eqnarray}
\end{rem}
We recall that if 
$T$ and $T_n$ (for all $n=1,2,...$) are bounded linear operators $A^p(\Omega)\rightarrow A^p(\Omega),$ the sequence of operators $(T_n)$ is said to converge strongly (in the strong operator topology, SOT) to $T$ iff $\left\Vert T_nf-Tf\right\Vert_p\rightarrow 0$ as $n\rightarrow \infty$ for all $f\in A^p(\Omega)$ (see for example \cite[p.83]{hutson}).
The next definition is needed in the condition for the symbol $a$.
\begin{defn}
	Given $S=S(u+iv,\rho)\subset \Omega$ and $z'=x'+iy' \in S,$ we write
	\[
	\hat{a}_S(z'):=\frac{1}{|S|}\int_{v}^{y'}\int_{u}^{x'}a(x+iy)dxdy.
	\]
	Note that
	if $z'=u+\rho+i(v+\rho),$ then $\hat{a}_S(z')$ is the average of $a$ over square $S(u+iv,\rho).$
\end{defn}
\noindent For a symbol $a\in L_{loc}^1(\Omega)$
we make always the assumption: There exists a constant $\mathbf{C}>0$ such that
\begin{equation}
	\label{symbol}
	|\hat{a}_S(z')|=\frac{1}{|S|}\left|\int_{v}^{y'}\int_{u}^{x'}a(x+iy)dxdy\right|\leq \mathbf{C}
\end{equation}
for all $z'=x'+iy'\in S:=S(u+iv, \rho)$ and for all squares $S\subset\Omega$ which have the property $\sqrt{2}\rho\leq\text{dist}(S,\partial\Omega)\leq 4\sqrt{2} \rho$. 

Recall that $\boldsymbol{\alpha_m}=\max_k({\alpha_k})$ and that $0<\alpha_k<2,$ $\alpha_k\neq 1,$ for all $k$. We state here the main result.
\begin{thm}
	\label{main1} Let $1<p<\infty$ and let $\Omega\subset \mathbb{C}$ be a polygon with corners $\alpha_1\pi,...,\alpha_n\pi$ at vertices $w_1,...,w_n\in \partial\Omega$. Suppose that $\boldsymbol{\alpha_m}<1+\frac{2}{p-2}$ if $p> 4$ and $\boldsymbol{\alpha_m}<1+\frac{2(p-1)}{2-p}$ if $1<p< 4/3$. Let $a \in L_{loc}^1(\Omega)$ and assume that symbol a satisfies the condition \eqref{symbol}. Then the generalized Toeplitz operator $T_{a,\Omega}$, defined as
	\begin{equation}
		\label{generalizedT2}
		(T_{a,\Omega}f)(z)=\lim_{m\rightarrow \infty}T^{(m)}_{a,\Omega}f(z)=\sum_{n=1}^{\infty}\int_{S_n}K_{\Omega}(z,w)a(w)f(w)dA(w),
	\end{equation}
	is a bounded operator from $A^p(\Omega)$ into $A^p(\Omega)$ and the sum in \eqref{generalizedT2} converges pointwise, absolutely for all $z \in \Omega.$ Moreover,  $T_{a,\Omega}^{(m)}\rightarrow 	T_{a,\Omega}$ strongly,  as $m\rightarrow \infty$.
\end{thm}
\begin{rem}
	Note that if $4/3\leq p \leq 4$, then there is no restriction for the maximum angle. The assumptions in Theorem \ref{main1} restrict only the maximum size of inward corners if $p>4$ or $1<p<4/3$ and they are based on the requirement of the Bergman projection to be bounded (see Theorem \ref{Regulated}).
\end{rem}
The article is organised as follows. In Section 2 we recall and reformulate some necessary estimates introduced in our predecessor \cite{mannersalo}. We also state some useful results concerning Bergman spaces and Bergman projections. 
The main result, Theorem \ref{main1}, is proved in Section 3. In section 4 we show a boundedness result in a weighted Bergman space $A_w^p(\Omega)$ in the case $p>4$ (Proposition \ref{main2}). We also deal with the case $1<p<4/3$ by setting a stronger condition for the symbol (Proposition \ref{main3}).
In Section 5 we consider the classical Toeplitz operator acting on Bergman space $A^p(\Omega)$  ($p>4$ or $1<p< 4/3$) in the case when the Bergman projection is unbounded. We show by examples that there are some difficulties in finding bounded Toeplitz operators in this special case.

\section{Preliminaries}
We  start this section with estimates for derivatives $\varphi^{(n)}$ ($n=1,2,3$) when $\varphi'$ is as in \eqref{schwarz2}. 
Let $\Omega$ be a polygon with vertices $w_1,...,w_n \in \partial \Omega$ and corners $\alpha_1\pi,...,\alpha_n\pi$. Let $\varphi:\Omega\rightarrow \mathbb{D}$ be a Riemann conformal mapping.
By Cauchy-Riemann equations ($w=x+iy$) we have
\begin{equation}
\label{partialder}
\partial_x{\overline{\varphi'(w)}}=\overline{\varphi''(w)}\;\;\;
\partial_y{\overline{\varphi'(w)}}=-i\overline{\varphi''(w)}\;\;\;
\partial_y\partial_x\overline{\varphi'(w)}
=-i\overline{\varphi'''(w)}.
\end{equation}
An elementary calculation leads to
\begin{align}
\label{derivatives2}
\nonumber
&\varphi'(w)= A^{-1}\prod_{k=1}^{n}G_{\Omega}(w, w_k)^{1-\alpha_k},\\
&\nonumber\varphi''(w)= A^{-2}\sum_{k=1}^n\prod_{\substack{j=1\\j\neq k}}^n(\alpha_k-1)\overline{\varphi(w_k)}G_{\Omega}(w,w_k)^{1-2\alpha_k}G_{\Omega}(w,w_j)^{2(1-\alpha_j)},\\
&\nonumber\varphi'''(w)= A^{-3}\sum_{k=1}^n\Big(\prod_{\substack{j=1\\j\neq k}}^nc_kG_{\Omega}(w,w_k)^{1-3\alpha_k}G_{\Omega}(w,w_j)^{3(1-\alpha_j)}\\
&\quad \quad \quad \quad +\sum_{\substack{j=1\\j\neq k}}^n\prod_{i\neq k,j}c_{k,j}G_{\Omega}(w,w_k)^{2-3\alpha_k}G_{\Omega}(w,w_j)^{2-3\alpha_j}G_{\Omega}(w,w_i)^{3(1-\alpha_i)}\Big),
\end{align}
where $G_{\Omega}(w,w_k)=1-\varphi(w)\overline{\varphi(w_k)}$, $c_k=(1-\alpha_k)(1-2\alpha_k)(\overline{\varphi(w_k)})^2$ and $c_{k,j}=2(1-\alpha_k)(1-\alpha_j)\overline{\varphi(w_k)}\overline{\varphi(w_j)}$.
To simplify evaluations including $|\varphi^{(n)}|$ we make the following observations.
Let $r=\min_{j,k}|w_j-w_k|/10>0$. We denote by $B(w_k,r)$ the disk with a center $w_k\in \partial \Omega$ and the radius $r$.
Let $\{\tilde{S_i}\}_i$ be as defined in \eqref{tildeS}, so $\Omega=\cup_{i} \tilde{S_i}$. The mapping $\varphi:\Omega\rightarrow \mathbb{D}$ has a homeomorphic extension $\varphi:\overline{\Omega}\rightarrow\overline{\mathbb{D}}$, since the boundary of $\Omega$ is obviously locally connected. 
Hence, there exist constants $C_m>0$ ($m=1,2,3$) such that 
\begin{equation}
\label{intersection}
C^{-1}_m<|\varphi^{(m)}(w)|< C_m\;\;\;\text{for all}\;\;\; w\in \bigcup_i\left\{\tilde{S_i}\;|\;\tilde{S_i}\subset \Omega\setminus\cup_{k=1}^nB(w_k,r)\right\},\;\;m=1,2,3.
\end{equation}
Let us fix $k$. Then for all $w\in \cup_i\{\tilde{S_i}\;|\;\tilde{S_i}\cap B(w_k,r)\neq \emptyset\}$
\begin{equation}
\label{intersection2}
|\varphi'(w)|\sim |G_{\Omega}(w,w_k)|^{1-\alpha_k},\;\;
|\varphi''(w)|\sim  |G_{\Omega}(w,w_k)|^{1-2\alpha_k},\;\;
|\varphi'''(w)|\sim  |G_{\Omega}(w,w_k)|^{1-3\alpha_k}.
\end{equation}
Let us make sure that the estimates in \eqref{intersection2} hold. If $\tilde{S_i}\cap B(w_k,r)\neq \emptyset$, then
for all $w\in \tilde{S_i}$
\[
|w-w_j|\geq \left||w-w_k|-|w_k-w_j|\right|\geq \left|41/19r-10r\right|=149/19r\;\;\;\text{for all}\;\;\;j\neq k,
\] 
where $|w_k-w_j|\geq 10r$ and $|w-w_k|\leq 11/10\sqrt{2}\rho_i+r\leq 11/10\text{dist}(S_i,\partial \Omega)+r\leq41/19 r$ (see Lemma \ref{whitney} (c) and \eqref{tildeS}). 
Therefore, since 
$\varphi:\overline{\Omega}\rightarrow \overline{\mathbb{D}}$ is a homeomorphism, there exists a constant $R>0$ such that
\[
R^{-1}< |G_{\Omega}(w,w_j)|^{1-\alpha_j}< R\;\;\;\text{for all}\;\;w\in \bigcup_i\left\{\tilde{S_i}\;|\;\tilde{S_i}\cap B(w_k,r)\neq \emptyset\right\},\;\;\text{for all}\;\;j\neq k.
\]
This implies
$|\varphi'(w)|=|A|^{-1}\prod_{j=1}^{n}|G_{\Omega}(w,w_j)|^{1-\alpha_j}\sim |G_{\Omega}(w,w_k)|^{1-\alpha_k}$ for all $w\in \bigcup_i\{\tilde{S_i}\;|\;\tilde{S_i}\cap B(w_k,r)\neq \emptyset\}.$
The estimates for $|\varphi''(w)|$ and $|\varphi'''(w)|$ in \eqref{intersection2} can be shown in the same way.

The following simple lemma is a modification of Lemma 2.3 in \cite{mannersalo}. 
\begin{lem}
	Let $\{\tilde{S}_n\}_n$ be the set of squares defined in \eqref{tildeS}. Then for all $n$
	\begin{align}
	\label{rho_1}
	&		\rho_n \sim \mathrm{dist}(w,\partial\Omega)\;\sim\; \frac{1-|\varphi(w)|^2}{|\varphi'(w)|}\;\;\;\text{for all}\;\;w\in \tilde{S}_n,\\
	\label{rho_4}
	&	\rho_n|\varphi'(w)|\leq C|1-\varphi(z)\overline{\varphi(w)}|\;\;\quad\quad\text{for all}\;\;w \in \tilde{S}_n\;\text{and all}\;\;z\in \Omega.
	\end{align}
	Moreover, if $\varphi'$ is as in \eqref{schwarz2}, then for all $n$ and all $w\in \tilde{S}_n$
	\begin{equation}
	\label{rho_5}
	\rho_n|\varphi''(w)| \leq C|\varphi'(w)|\;\;\;\;\text{and}\;\;\;\;
	\rho^2_n|\varphi'''(w)|\leq C'|\varphi'(w)|.
	\end{equation}
\end{lem}
\begin{proof} The relation \eqref{rho_1} follows from \eqref{weight} and Lemma \ref{whitney} (c).
	The inequality \eqref{rho_4} holds since  $|1-\varphi(z)\overline{\varphi(w)}|\geq 1-|\varphi(w)|\geq C' \rho_n|\varphi'(w)|$ according to \eqref{rho_1}. We show that the latter inequality in \eqref{rho_5} holds, the former follows by similar arguments. Indeed, if $\tilde{S}_n\cap B(w_k,r)\neq \emptyset$ (where $r=\min_{i,j}|w_i-w_j|/10$) for some $k$, then due to \eqref{intersection2}
	\begin{eqnarray}
	\nonumber\rho_n^2|\varphi'''(w)|&\leq& C\frac{(1-|\varphi(w)|^2)^2|\varphi'''(w)|}{|\varphi'(w)|^2}\leq C'\frac{(1-|\varphi(w)|^2)^2|G_{\Omega}(w,w_k)|^{1-3\alpha_k}}{|G_{\Omega}(w,w_k)|^{2-2\alpha_k}}
	\\
	\nonumber
	&=& C'\frac{(1-|\varphi(w)|^2)^2}{|G_{\Omega}(w,w_k)|^{1+\alpha_k}}
	\leq C''|G_{\Omega}(w,w_k)|^{1-\alpha_k}\leq C'''|\varphi'(w)|,
	\end{eqnarray}
	since always $1-|\varphi(w)|^2\leq 2|1-\varphi(w)\overline{\varphi(w_k)}|=2|G_{\Omega}(w,w_k)|$.
	If $\tilde{S}_n\cap B(w_k,r)= \emptyset$ for all $k$, then the inequality $\rho_n^2|\varphi'''(w)|\leq C|\varphi'(w)|$ follows from \eqref{intersection}.
\end{proof}
We recall here some useful results regarding to the boundedness of the Bergman projections.
The maximal Bergman projection of the unit disk $P^+_{\mathbb{D}}$ is bounded on $L_{\omega}^p(\mathbb{D})$ (weight $\omega$ is positive, locally integrable) if and only if the Bergman projection $P_{\mathbb{D}}:L_{\omega}^p(\mathbb{D})\rightarrow A_{\omega}^p(\mathbb{D})$ is bounded \cite{bekolle1}. It is easy to see (by changing variables) that $P_{\Omega}$ defines a bounded projection $L^p(\Omega)\rightarrow A^p(\Omega)$ if and only if $P_{\mathbb{D}}$ defines a bounded projection $L_{\omega}^p(\mathbb{D})\rightarrow A_{\omega}^p(\mathbb{D})$, where the weight is $\omega(z)=|\psi'(z)|^{2-p}$ \cite{hedenmalm,shikhvatov}. 
These facts imply that
the boundedness of $P^+_{\Omega}:L^p(\Omega)\rightarrow L^p(\Omega)$ is equivalent to the boundedness of  $P_{\Omega}:L^p(\Omega)\rightarrow A^p(\Omega)$, see \cite[Theorem 3.1]{lanzani}.

For results concerning the boundedness of the Bergman projection on simply connected planar domains we refer to \cite{solovev, solovev2, hedenmalm, bekolle2, shikhvatov, taskinen3, lanzani}.
The next theorem gives a relation between the geometry of a polygonal domain and the boundedness of the Bergman projection. It is a partial result (restricted to polygons) of more general results of \cite{bekolle2,solovev, solovev2,taskinen3, taskinen4}. Those results were proven by showing that the boundedness of the Bergman projection is equivalent to the B\'{e}koll\'{e}-Bonami condition for the weight $|\psi|^{2-p}$.
\begin{thm}
	\label{Regulated}
	Let $\Omega \subset \mathbb{C}$ be a polygon with the maximum angle $\boldsymbol{\alpha_{m}}\pi$. The Bergman projection $P_{\Omega}$ defines a bounded projection from $L^p(\Omega)$ onto $A^p(\Omega)$ if and only if
	\begin{equation}
	\label{inward1}
	(2-p)(\boldsymbol{\alpha_m}-1)<2(p-1)\;\;\text{in the case}\;\; p\leq 2,\;\;\text{or,}\;\;(p-2)(\boldsymbol{\alpha_m}-1)<2\;\;\text{in the case}\;\;p\geq 2.
	\end{equation}
\end{thm}
\noindent For the proof, see \cite[Th\'{e}or\`{e}me 2.1]{bekolle2}  or \cite[Theorem 4.1]{taskinen3} and Remarks in \cite[p.99]{taskinen4}. Th\'{e}or\`{e}me 2.1 in \cite{bekolle2} deals with Lavrentiev domains (for the definition, see \cite[p.163--164]{pommerenke}).
The papers \cite{taskinen3} and  \cite{taskinen4} deal with regulated domains (for the definition, see \cite[p.59--60]{pommerenke}).
\begin{rem}
	Notice that there is no restriction for outward pointing corners ($0<\alpha<1$), i.e., if $\Omega$ is a polygon with only outward corners, then $P_{\Omega}$ is bounded on $L^p(\Omega)$ for all $1<p<\infty$. We also point out that if $\frac{4}{3}\leq p\leq 4,$ then $P_{\Omega}$ is bounded on $L^p(\Omega)$ for every polygon $\Omega.$
\end{rem}
\noindent The next corollary is an obvious consequence of Theorem \ref{Regulated}.
\begin{cor}
	\label{unit} Let $\Omega$ be a polygon with the maximum angle $\boldsymbol{\alpha_{m}}\pi$.
Let $\psi':\mathbb{D}\rightarrow \Omega$ be as \eqref{schwarz1} and $\omega(z)=|\psi'(z)|^{2-p}$ a weight on space $L^p_{\omega}(\mathbb{D})$. If $0<\boldsymbol{\alpha_{m}}<1$, then $P_{\mathbb{D}}$ defines a bounded projection $L^p_{\omega}(\mathbb{D})\rightarrow A^p_{\omega}(\mathbb{D})$ for all $1<p<\infty$. If $1<\boldsymbol{\alpha_{m}}<2$, then $P_{\mathbb{D}}$ defines a bounded projection $L^p_{\omega}(\mathbb{D})\rightarrow A^p_{\omega}(\mathbb{D})$ if and only if
	\begin{equation}
	\label{unit1}
	(2-p)(\boldsymbol{\alpha_m}-1)<2(p-1)\;\;\text{in the case}\;\; p\leq 2,\;\;\text{or,}\;\;(p-2)(\boldsymbol{\alpha_m}-1)<2\;\;\text{in the case}\;\;p\geq 2.
	\end{equation}
\end{cor}

Let $f\in A^p(\Omega)$. Let $S_n$ be a Whitney square and $\tilde{S}_n\supset S_n$, see \eqref{Sn} and \eqref{tildeS}. Due to the mean value property for analytic functions, we have
\begin{equation}
\label{meanvalue}
|f(z)|\leq \frac{C}{|S_n|}\int_{\tilde{S}_n}|f(w)|dA(w)
\end{equation}
for all $z\in S_n$ (for details, see \cite{mannersalo}).

The following norm estimate follows from the proof of Theorem 4.28 in \cite{zhu}.
\begin{lem}\emph{\cite{zhu}}
	\label{norms}
	Let $n\in \mathbb{Z_+}$. There exists positive constant $C_n$ such that 
	\begin{equation}
	\label{norms1}	
	 \left\Vert(1-|z|^2)^nf^{(n)}(z)\right\Vert_{p,\mathbb{D}}\leq C_n\left\Vert f \right\Vert_{p,\mathbb{D}}\;\;\;\text{for all}\;\;\;f\in A^p(\mathbb{D}).
	\end{equation}
\end{lem}
\noindent 
We need a generalization of Lemma \ref{norms} (when $n=1$ or $n=2$) to the case of $A^p(\Omega),$ where $\Omega$ is a polygon (compare this to \cite[Corollary 2.5]{mannersalo}):
\begin{lem}
	\label{corderivative}
	Let $\Omega\subset \mathbb{C}$ be a polygon. There exist $C_p>0$ and $C'_p>0$ such that for all $f \in A^p(\Omega)$ \begin{equation}
	\label{cor1}
	\left\Vert\mathrm{dist}(w,\partial\Omega)f'(w)\right\Vert_p	\leq C_p\left\Vert f\right\Vert_p
	\;\;\;\;\text{and}\;\;\;\;
	\left\Vert (\mathrm{dist}(w,\partial\Omega))^2f''(w)\right\Vert_p\leq C_p' \left\Vert f\right\Vert_p. \end{equation}
\end{lem}
\begin{proof}
Since $\text{dist}(w,\partial \Omega)\sim (1-|\varphi(w)|^2)|\varphi'(w)|^{-1}$ according to \eqref{weight}, we obtain by changing variables $z=\varphi(w)$ ($w=\psi(z)$, Jacobian $|\psi'|^2$) and using the triangle inequality
\begin{align}
\nonumber\left\Vert\mathrm{dist}(w,\partial\Omega)f'(w)\right\Vert_p&\leq C\left\Vert  (1-|\varphi(w)|^2)|\varphi'(w)|^{-1}f'(w)\right\Vert_p= C\left\Vert  (1-|z|^2)(\psi'(z))^{2/p+1}f'(\psi(z))\right\Vert_{p,\mathbb{D}}\\
&\nonumber=C\left\Vert  (1-|z|^2)\Big(D_z\big((f\circ \psi)(\psi')^{2/p}\big)-\frac{2}{p}(f\circ \psi)(\psi')^{2/p-1}\psi''\Big)\right\Vert_{p,\mathbb{D}}\\
&\nonumber\leq C_p\left\Vert  (1-|z|^2)D_z\big((f\circ \psi)(\psi')^{2/p}\big)\right\Vert_{p,\mathbb{D}}+C_p\left\Vert(1-|z|^2)(f\circ \psi)(\psi')^{2/p-1}\psi''\right\Vert_{p,\mathbb{D}}\\
&\label{third}\leq C_p'\left\Vert (f\circ \psi)(\psi')^{2/p}\right\Vert_{p,\mathbb{D}}=C_p'||f||_{p},
\end{align}
where $D_z$ is the derivative operator. We have applied \eqref{norms1} with $n=1$ to the first term on the second to last line. The second term is also bounded above by $\left\Vert (f\circ \psi)(\psi')^{2/p}\right\Vert_{p,\mathbb{D}}$ since for all $z\in \mathbb{D}$ ($z=\varphi(w)$) \begin{equation}
\label{fourth}
\left|(1-|z|^2)(\psi'(z))^{-1}\psi''(z)\right|=\left|(1-|\varphi(w)|^2)(\varphi'(w))^{-2}\varphi''(w)\right|\leq C'\end{equation} 
because of \eqref{intersection} and \eqref{intersection2}. The second norm estimate in \eqref{cor1} follows by similar arguments:
\begin{align*}
\left\Vert(\mathrm{dist}(w,\partial\Omega))^2f''(w)\right\Vert_p&\leq C\left\Vert  (1-|\varphi(w)|^2)^2|\varphi'(w)|^{-2}f''(w)\right\Vert_p\\&= C\left\Vert  (1-|z|^2)^2(\psi'(z))^{2/p+2}f''(\psi(z))\right\Vert_{p,\mathbb{D}}\\
&\leq C_p'\left\Vert (f\circ \psi)(\psi')^{2/p}  \right\Vert_{p,\mathbb{D}}=C_p'||f||_{p},
\end{align*}
where we have taken into account that
\begin{align}
\nonumber
(1-|z|^2)^2(\psi')^{2/p+2}f''(\psi(z))&=(1-|z|^2)^2\Big(D^2_z\big((f\circ\psi)(\psi')^{2/p}\big)-(4/p+1)(\psi')^{2/p}\psi''f'(\psi(z))
\\&-2/p(2/p-1)(f\circ\psi)(\psi')^{2/p-2}(\psi'')^2-2/p(f\circ\psi)(\psi')^{2/p-1}\psi'''\Big).\label{second}
\end{align}
The norm of each term of \eqref{second} is bounded above by $C_p''\left\Vert  (f\circ \psi)(\psi')^{2/p} \right\Vert_{p,\mathbb{D}}$: the first term due to \eqref{norms1} with $n=2$, the second term because of \eqref{third} and \eqref{fourth}, the third term since $|(1-|z|^2)^2(\psi'(z))^{-2}(\psi''(z))^2|\leq C''$ for all $z$ and the last term since $|(1-|z|^2)^2(\psi'(z))^{-1}\psi'''(z)|\leq C'''$ (these conclusions can be made by considering directly the formula of $\psi'$, see \eqref{schwarz1}, and its derivatives $\psi'',\psi'''$ or by using $\varphi',\varphi'',\varphi'''$, see \eqref{derivatives2}, \eqref{intersection} and \eqref{intersection2}).
\end{proof}
\noindent
The next lemma relates the Taylor coefficients of an $A^p(\mathbb{D})$-function with its norm \cite[p.83--84]{duren}.
\begin{lem}\emph{\cite{duren}}
	\label{taylor}
	Let $2<p<\infty$ and $f\in A^p(\mathbb{D})$. Assume that $\sum_{n=0}^{\infty}a_nz^n$ is the Taylor series of $f.$ Then there exists $C_p>0$ such that
	\begin{equation}
	\label{taylorcof}
	C^{-1}_p\sum_{n=0}^{\infty}(n+1)^{-1}|a_n|^p\leq \left\Vert f\right\Vert^p_p\leq C_p\sum_{n=0}^{\infty}(n+1)^{p-3}|a_n|^p.
	\end{equation}
	Moreover, the exponents of $n+1$ are the best possible.
\end{lem}
In the last section we need the following mapping properties of the Bergman projection $P_{\mathbb{D}}$
(for details and proofs, see \cite[p.~46--47]{duren}).
\begin{lem}\emph{\cite{duren}}
	\label{bergman_prop}
	Let $P_{\mathbb{D}}$ be the Bergman projection of the unit disk, and let $m$ and $n$ be nonnegative integers. Then
		\begin{align}
		\label{mapa}
		& P_{\mathbb{D}}((1-|z|^2)z^m) = \frac{z^m}{m+2},\\
		\label{mapb}
		& P_{\mathbb{D}}(z^m\overline{z}^n) = \left\{\begin{array}{ll}
		\frac{m-n+1}{m+1}z^{m-n}, &\;\;\;\text{for all}\;\; m\geq n \\
		0, &\;\;\;\text{for all}\;\; m< n,
		\end{array} \right.\\
		\label{mapc}
		&P_{\mathbb{D}}((1-|z|^2)z^{m}\overline{z}^n) =  \left\{\begin{array}{ll}
		\frac{m-n+1}{(m+1)(m+2)}z^{m-n}, &\;\;\text{for all}\;\; m\geq n \\
		0, &\;\;\text{for all}\;\; m< n.
		\end{array} \right.
		\end{align}
\end{lem}
\noindent Note that the property \eqref{mapc} follows from \eqref{mapa} and \eqref{mapb}.
Finally, we recall the definition of the Cauchy-Szegö integral operator $S$. 
Let $f\in L^1(\partial\mathbb{D})$ and
\begin{equation}
\label{szego}
(Sf)(z):=\frac{1}{2\pi}\int_{0}^{2\pi}\frac{f(e^{i\theta})}{1-ze^{-i\theta}}d\theta, \;\;\;z\in \mathbb{D}.
\end{equation}
Due to the Cauchy integral formula, $S$ has the reproducing property, i.e., if $f:\overline{\mathbb{D}}\rightarrow \mathbb{C}$ is continuous on $\overline{\mathbb{D}}$ and analytic on $\mathbb{D}$, then $(Sf)(z)=f(z)$. 	 
\section{Proof of Theorem \ref{main1}}
Before the actual proof of the main theorem we make a few more simplifications in the notation. Let $z\in \Omega$ and $S_n=S_n(x_n+iy_n,\rho_n)$ (see \eqref{Sn}). We denote
\begin{equation}
\label{Fn}
F_nf(z):=\int_{S_n}K_{\Omega}(z,w)a(w)f(w)dA(w),
\end{equation}
so that (see \eqref{generalizedT}) 
\begin{equation}
\label{Fnomega}
(T_{a,\Omega}f)(z)=\sum_{n=1}^{\infty}F_{n}f(z)\;\;\;\text{for all}\;\;\;z\in \Omega.
\end{equation}
Integration by parts as in \cite{mannersalo} yields
\begin{equation}
\label{Fnexpression}
F_nf(z)=F_{1,n}f(z)-F_{2,n}f(z)-F_{3,n}f(z)+F_{4,n}f(z),
\end{equation}
where
\begin{equation}
\label{F1omega}
F_{1,n}f(z):=\left(\int_{y_n}^{y_n'}\int_{x_n}^{x_n'}a(x+iy)dxdy\right)f(x_n'+iy_n')K_{\Omega}(z,x_n'+iy_n'),
\end{equation}
\begin{equation}
\label{F2omega}
F_{2,n}f(z):=\intop_{y_n}^{y_n'}
\left(\int_{y_n}^{y}\int_{x_n}^{x_n'}a(x+iy)dxdy\right)
\partial_yf(x_n'+iy)K_{\Omega}(z,x_n'+iy)dy,
\end{equation}
\begin{equation}
\label{F3omega}
F_{3,n}f(z):=\intop_{x_n}^{x_n'}\left(\intop_{y_n}^{y_n'}\intop_{x_n}^{x}a(x+iy)dxdy \right)
\partial_{x}f(x+iy_n')K_{\Omega}(z,x+iy_n')dx,
\end{equation}
\begin{equation}
\label{F4omega}
F_{4,n}f(z):=\intop_{x_n}^{x_n'}\intop_{y_n}^{y_n'}\left(\intop_{y_n}^{y}\intop_{x_n}^{x}a(x+iy)dxdy \right)
\partial_{y}\partial_{x}f(x+iy)K_{\Omega}(z,x+iy)dydx,
\end{equation}
where $x_n'+iy_n'=(x_n+\rho_n)+i(y_n+\rho_n)\in S_n$ in the notation \eqref{Sn} and $K_{\Omega}$ is the Bergman kernel \eqref{kernel}.

We now form bounds for $|F_{k,n}f(z)|$ ($k=1,2,3,4$), for details see \cite[p.862--863]{mannersalo}. We only present the calculations for $|F_{1,n}f|$ and $|F_{4,n}f|.$ 
Let us fix $z\in \Omega$ and define an analytic function
\[
h_z(w):=\frac{\overline{\varphi'(z)}\varphi'(w)f(w)}{(1-\overline{\varphi(z)}\varphi(w))^2},\;\;w\in \Omega.
\] 
The estimate for $|F_{1,n}f(z)|$ is now obtained by applying \eqref{symbol} and \eqref{meanvalue} to the function $h_z$. We denote $z'_n=x'_n+iy'_n$. Thus, for all $z\in \Omega$
\begin{align}
\label{F1}
|F_{1,n}f(z)|&\leq\left|\int_{y_n}^{y_n'}\int_{x_n}^{x_n'}a(x+iy)dxdy\right|\frac{|f(z_n')||\varphi'(z)||\overline{\varphi'(z_n')}|}{|1-\varphi(z)\overline{\varphi(z_n')}|^2}
\nonumber
\\
&\nonumber=\left|\int_{y_n}^{y_n'}\int_{x_n}^{x_n'}a(x+iy)dxdy\right||h_z(z_n')|\leq
\mathbf{C}|S_n|\frac{C}{|S_n|}\int_{\tilde{S}_n}|h_z(w)|dA(w)
\\
&=
\mathbf{C}|S_n|\frac{C}{|S_n|}\int_{\tilde{S}_n}\frac{|\varphi'(z)||\overline{\varphi'(w)}||f(w)|}{|1-\varphi(z)\overline{\varphi(w)}|^2}dA(w) 
=C'\int_{\tilde{S_n}}|K_{\Omega}(z,w)||f(w)|dA(w).
\end{align}
For $|F_{k,n}f(z)|$, $k=2,3,$ we obtain the bound
\begin{align}
\label{F3}
|F_{k,n}f(z)|&\leq C''
\int_{\tilde{S}_n}|K_{\Omega}(z,w)|\Big(v(w)|f'(w)|+|f(w)|\Big)
dA(w),
\end{align}
where $v(w)=\text{dist}(w,\partial \Omega)$.
The inequality \eqref{F3} is based on \eqref{schwarz2}, \eqref{symbol}, \eqref{meanvalue} and the estimates
$\rho_n\sim (1-|\varphi(w)|^2)|\varphi'(w)|^{-1}\sim v(w)$ (see \eqref{rho_1}), $\rho_n|\varphi'(w)|\leq C|1-\varphi(z)\overline{\varphi(w)}|$ (see \eqref{rho_4}) and $\rho_n|\varphi''(w)|\leq C'|\varphi'(w)|$ (see \eqref{rho_5}).
\allowdisplaybreaks

Applying \eqref{symbol} and \eqref{meanvalue} (to analytic functions which vary in terms of the following sum) gives
the bound for $|F_{4,n}f(z)|$ as follows. 
Recall the notation 
$G_{\Omega}(z,w)=1-\varphi(z)\overline{\varphi(w)}.$ For all $z\in \Omega$ ($w=x+iy$)
\begin{align}
\nonumber
|F_{4,n}f(z)|
&\leq  |\varphi'(z)| \intop_{x_n}^{x_n'}\intop_{y_n}^{y_n'}\left|\intop_{y_n}^{y}\intop_{x_n}^{x}a(w)dxdy \right|
\left|\partial_{y}\partial_{x}\frac{f(w)\overline{\varphi'(w)}}{(1-\varphi(z)\overline{\varphi(w)})^2}\right|dydx
\\
\nonumber
&\leq
C'|\varphi'(z)|
\int_{\tilde{S}_n}|G_{\Omega}(z,w)|^{-2}
\left(
\rho_n^2|\overline{\varphi'(w)}||f''(w)|
+
\rho_n^2|\overline{\varphi''(w)}||f'(w)|
+
\frac{\rho_n^2|\varphi(z)||\overline{\varphi'(w)}|^2|f'(w)|}{|1-\varphi(z)\overline{\varphi(w)}|}
\right.
\\
\nonumber
&+
\left.
\rho_n^2|\overline{\varphi''(w)}||f'(w)|
+
\rho_n^2|\overline{\varphi'''(w)}||f(w)|
+
\frac{\rho_n^2|\varphi(z)||\overline{\varphi'(w)}||\overline{\varphi''(w)}||f(w)|}{|1-\varphi(z)\overline{\varphi(w)}|}
\right.
\\
\nonumber
&+
\left.
\frac{\rho_n^2|\varphi(z)||\overline{\varphi'(w)}|^2|f'(w)|}{|1-\varphi(z)\overline{\varphi(w)}|}
+
\frac{\rho_n^2|\varphi(z)||\overline{\varphi'(w)}||\overline{\varphi''(w)}||f(w)|}{|1-\varphi(z)\overline{\varphi(w)}|}
+
\frac{\rho_n^2|\varphi(z)||\overline{\varphi'(w)}|^3|f(w)|}{|1-\varphi(z)\overline{\varphi(w)}|^2}
\right)dA(w)
\\
\label{F4}
\nonumber
&\leq C''\int_{\tilde{S}_n}
|\varphi'(z)||\overline{\varphi'(w)}||G_{\Omega}(z,w)|^{-2}\Big(v(w)^2|f''(w)|+v(w)|f'(w)|+|f(w)|\Big)dA(w)
\\
&\leq
C'''
\int_{\tilde{S}_n}
|K_{\Omega}(z,w)|\Big(v(w)^2|f''(w)|+v(w)|f'(w)|+|f(w)|\Big)dA(w),
\end{align}
where we have applied the estimates $\rho_n\sim (1-|\varphi(w)|^2)|\varphi'(w)|^{-1}\sim v(w)$,
$\rho_n|\varphi'(w)|\leq C|1-\varphi(z)\overline{\varphi(w)}|$ (see \eqref{rho_4}),
$\rho_n|\varphi''(w)|<C|\varphi'(w)|$ and $\rho_n^2|\varphi'''(w)|<C|\varphi'(w)|$ (see \eqref{rho_5}).

By combining the estimates \eqref{F1}--\eqref{F4} we obtain: there exists a constant $C>0$ such that for all $n$ and all $z \in \Omega$
\begin{equation}
\label{Fnestimate}
|F_nf(z)|\leq C\int_{\tilde{S}_n}
|K_{\Omega}(z,w)|\Big(|f(w)|+v(w)|f'(w)|+v(w)^2|f''(w)|\Big)dA(w),
\end{equation}
where $v(w)=\text{dist}(w,\partial\Omega)$.
\paragraph{Proof of Theorem \ref{main1}}
\begin{proof}
	Let $f \in A^p(\Omega)$ and $v(w)=\text{dist}(w,\partial\Omega).$  Recall the definition of $F_nf(z)$ (see \eqref{Fn} and \eqref{Fnomega}). We obtain
	\begin{align}
	\nonumber
	\left|(T_{a,\Omega}f)(z)\right|&=\left|\sum_{n=1}^{\infty} \int_{S_n}K_{\Omega}(z,w)f(w)a(w)dA(w)\right|=\left|\sum_n F_{n}f(z)\right|
	\leq \sum_n |F_{n}f(z)|
	\\
	\nonumber
	&\leq
	144C \int_{\Omega}\left|K_{\Omega}(z,w)\right|\Big(|f(w)|+|f'(w)|v(w)+|f''(w)|v(w)^2\Big)dA(w)\\\
	\label{finite}
	&=
	144C P_{\Omega}^+\big(|f|+|f'|v+|f''|v^2\big)(z)< \infty,
	\end{align}
	where we have applied \eqref{Fnestimate}, and the domain $\Omega$ has been covered by partly overlapping squares $\tilde{S}_n$, see Lemma \ref{whitney} (d). The integral on the second line is finite: $|1-\varphi(z)\overline{\varphi(w)}| \geq 1-|\varphi(z)|$ for all $w\in \Omega$ and by Hölder's inequality we have ($1/p+1/q=1$)
	\begin{eqnarray}
	\nonumber
	\label{hölder}
	\int_{\Omega}\left|f(w)\right|\left|\varphi'(w)\right|dA(w) &\leq&
	\left(\int_{\Omega}|f(w)|^pdA(w)\right)^{1/p}\left(\int_{\Omega}|\varphi'(w)|^qdA(w)\right)^{1/q}
	\\\nonumber
	&=&\left\Vert f\right\Vert_{p,\Omega}\left(\int_{\mathbb{D}}|\psi'(\xi)|^{2-q}dA(\xi)\right)^{1/q}
	\\
	&=&\left\Vert f\right\Vert_{p,\Omega}\left(\int_{\mathbb{D}}\big|\prod_{k=1}^{n}(1-\overline{\xi}_k\xi)^{\alpha_k-1}\big|^{(p-2)/(p-1)}dA(\xi)\right)^{1/q}<\infty,
	\end{eqnarray}
	where we have changed variables $w=\psi(\xi)$ and recalled the formula \eqref{schwarz1}. The convergence in \eqref{hölder} is based on the fact that integral of the form (for a fixed $\theta$) $\int_{\mathbb{D}}|1-e^{i\theta}\xi|^{-s}dA(\xi)<\infty$ if and only if $s<2$ \cite[p.78--79]{duren} (note that $\overline{\xi}_k\in \partial\mathbb{D}$ for all $k$). We have also recalled the assumption $\boldsymbol{\alpha_m}<1+2(p-1)/(2-p)$ if $1<p< 4/3$ in Theorem \ref{main1}, so that for all exponents $(\alpha_k-1)(p-2)/(p-1)>-2$ as is necessary. If $p\geq 4/3$, then $(\alpha_k-1)(p-2)/(p-1)>-2$ since $(p-2)/(p-1)\geq -2$ and always $-1<\alpha_k-1<1$.  
	According to Lemma \ref{corderivative}, $|f'|v$ and $|f''|v^2$ also belong to $L^p(\Omega)$. Thus, the integral in \eqref{finite} is clearly finite. This proves that the sum \eqref{generalizedT2} converges pointwise and absolutely. 
	
	Because of the assumptions for the maximum angle, the Bergman projection $P_{\Omega}$ is now bounded (see Theorem \ref{Regulated}).
	The maximal Bergman projection  $P_{\Omega}^{+}$ is bounded as well and we get
	\begin{align*}
	\left\Vert T_{a,\Omega}f\right\Vert_p
	&\leq 
	C'
	\left\Vert P_{\Omega}^+\big(\left|f\right|+\left|f'\right|v+\left|f''\right|v^2\big)
	\right\Vert_p
	\\
	&
	\leq C'
	\big(\left\Vert P_{\Omega}^+(\left|f\right|)\right\Vert_p+\left\Vert P_{\Omega}^+(\left|f'\right|v)\right\Vert_p+\left\Vert P_{\Omega}^+(\left|f''\right|v^2)\right\Vert_p\big)
	\leq C''\left\Vert f\right\Vert_p,
	\end{align*}
	where the last inequality follows from Lemma \ref{corderivative}. This proves the boundedness of $T_{a,\Omega}.$

	Next, we show that $T_{a,\Omega}^{(m)}\rightarrow T_{a,\Omega}$ strongly (in SOT) as $m\rightarrow \infty$. Let $f\in A^p(\Omega)$ and $z\in \Omega.$  Due to \eqref{Fn} and \eqref{Fnestimate}, we have
	\begin{align}
	\label{SOT}
	\nonumber
	&\left|(T_{a,\Omega}f)(z)-(T_{a,\Omega}^{(m)}f)(z)\right|=\left|\sum_{n=m+1}^{\infty}(F_nf)(z)\right|\leq\sum_{n=m+1}^{\infty}|(F_nf)(z)| \\
	&\quad\leq \int_{\bigcup_{n>m} \tilde{S}_n}\left|K_{\Omega}(z,w)\right|\big(|f(w)|+|f'(w)|v(w)+|f''(w)|v(w)^2\big)dA(w).
\end{align}
Let $\mbox{\large$\chi$}_{V_{m}}$ be the characteristic function of the set $\bigcup_{n>m} \tilde{S}_n=:V_{m}$. Since every $z\in \Omega$ belongs to at most 144 of the squares $\tilde{S}_n$ (see Lemma \ref{whitney} property (d)), obviously $\mbox{\large$\chi$}_{V_{m}}(z)\rightarrow 0$ as $m\rightarrow \infty$ for all $z$.
We can now continue from \eqref{SOT}:
	\begin{eqnarray}
	\label{SOT2}
	\nonumber
	\left|(T_{a,\Omega}f)(z)-(T_{a,\Omega}^{(m)}f)(z)\right| 
	\nonumber
	&\leq& \int_{\Omega}\mbox{\large$\chi$}_{V_{m}}(w)\left|K_{\Omega}(z,w)\right|\big(|f(w)|+|f'(w)|v(w)+|f''(w)|v(w)^2\big)dA(w)
	\\
	&=& P_{\Omega}^+\big(\mbox{\large$\chi$}_{V_{m}}(|f|+|f'|v+|f''|v^2)\big)(z)=P_{\Omega}^+(\mbox{\large$\chi$}_{V_{m}}g)(z),
	\end{eqnarray}
	where $g:=\left|f\right|+\left|f'\right|v+\left|f''\right|v^2.$
	It follows from the Lebesgue's dominated convergence theorem that $\left\Vert \chi_{V_{m}}g\right\Vert_p\rightarrow 0$ as $m\rightarrow \infty.$ Since the maximal Bergman projection $P_{\Omega}^+$ is bounded, also $\left\Vert P_{\Omega}^+(\chi_{V_{m}}g)\right\Vert_p\rightarrow 0$ as $m\rightarrow \infty.$ Applying this to the inequality \eqref{SOT2} implies the claim.
	
	Finally, we remark that $T_{a,\Omega}f$ is an analytic function since $F_nf$ are clearly analytic and $(T_{a,\Omega}^{(m)}f)(z)$ converges uniformly to $(T_{a,\Omega}f)(z)$ on compact subsets of $\Omega$ as $m\rightarrow \infty$. This follows from \eqref{SOT} by taking into account that $|\varphi'(z)|/|1-\varphi(z)\overline{\varphi(w)}|^2\leq|\varphi'(z)|/(1-|\varphi(z)|)^2$, where the right-hand side is bounded on compact subset of $\Omega$. 
	Moreover, the integral $\int_{\Omega}|\varphi'(w)|\big(|f(w)|+|f'(w)|v(w)+|f''(w)|v(w)^2\big)dA(w)$ is finite, see  \eqref{hölder} and the comments after it. This completes the proof.
\end{proof}
\section{Boundedness results in the case of unbounded Bergman projection}
If $p> 4$ or $p< 4/3$ and polygon $\Omega$ has big enough inward corners ($1<\alpha<2)$, then the Bergman projection $P_{\Omega}$ is unbounded  according to Theorem \ref{Regulated}. In this special case the situation is more complicated as we will see in the last section. However, when $p>4,$ it is possible to prove a boundedness result in a weighted Bergman space $A_{\omega}^p(\Omega)$ with the weight  $w(w)=|G_{\Omega}(w,\boldsymbol{w_m})|^{t}$ (Proposition \ref{main2}). In the case $p<4/3$, we obtain a bounded Toeplitz operator by strengthening the condition \eqref{symbol}. In addition to boundedness, we require the ''average'' of symbol, $|\hat{a}_S(z')|$, to converge to zero (at pace $|G_{\Omega}(z',\boldsymbol{w_m})|^t$)
when approaching the vertex $\boldsymbol{w_{m}}$ (Proposition \ref{main3}).
In both cases, we deal with a polygon that has only one ''tricky'' corner.
Recall that $\boldsymbol{\alpha_m}=\max_{k}(\alpha_k)$ and $G_{\Omega}(z,w)=1-\varphi(z)\overline{\varphi(w)}.$
\begin{prop}
	\label{main2}
	Let $p>4$ and let $\Omega\subset \mathbb{C}$ be a polygon with corners $\alpha_1\pi,...,\alpha_n\pi$ $(0<\alpha_k<2,$ $\alpha_k\neq 1)$ at vertices $w_1,...,w_n\in \partial\Omega$.
	 Suppose that $\boldsymbol{\alpha_m}\geq 1+\frac{2}{p-2}$ and $\alpha_k< 1+\frac{2}{p-2}$ for all $k\neq \boldsymbol{m}$.
	 Let $a \in L_{loc}^1(\Omega)$ and assume that the condition \eqref{symbol} holds.
	Then the sum \eqref{generalizedT2} converges absolutely for all $z \in \Omega$ and the generalized Toeplitz operator $T_{a,\Omega}$
	is a bounded operator from $A_{\omega}^p(\Omega)$ into $A_{\omega}^p(\Omega)$, where  $\omega(w)=|G_{\Omega}(w, \boldsymbol{w_m})|^{t}$ and $t> (p-2)(\boldsymbol{\alpha_m}-1)-2$.	
\end{prop}
\begin{prop}
	\label{main3}
	Let $1<p<4/3$ and let $\Omega\subset \mathbb{C}$ be a polygon with corners $\alpha_1\pi,...,\alpha_n\pi$ $(0<\alpha_k<2,$ $\alpha_k\neq 1)$ at vertices $w_1,...,w_n\in \partial\Omega$.
	Suppose that
	$\boldsymbol{\alpha_m}\geq 1+\frac{2(p-1)}{2-p}$ and $\alpha_k< 1+\frac{2(p-1)}{2-p}$ for all $k\neq \boldsymbol{m}$. 
	Let $a \in L_{loc}^1(\Omega)$ and assume that 
	$t> (2-p)(\boldsymbol{\alpha_m}-1)-2(p-1)$ and 
	\begin{equation}
	\label{symbol3}
	|\hat{a}_S(z')|\leq \mathbf{C}|G_{\Omega}(z',\boldsymbol{w_m})|^{t}
	\end{equation}
	for all $z'\in S$ and all squares $S\subset\Omega$ with  $\sqrt{2}\rho\leq\text{dist}(S,\partial\Omega)\leq 4\sqrt{2} \rho$. 
	Then the sum \eqref{generalizedT2} converges absolutely for all $z \in \Omega$ and the generalized Toeplitz operator $T_{a,\Omega}$
	is a bounded operator from $A^p(\Omega)$ into $A^p(\Omega)$.	
\end{prop}
\noindent
The proofs are similar to that of Theorem \ref{main1}, so we state only briefly the main parts and changes.
\paragraph{Proof of Proposition \ref{main2}}
\begin{proof} 
	 The crucial idea is that with the aid of the weight $\omega(w)=|1-\varphi(w)\overline{\varphi(\boldsymbol{w_{m}})}|^{t},$ $t>(p-2)(\boldsymbol{\alpha_m}-1)-2>0$, we obtain a bounded maximal Bergman projection $P_{\mathbb{D}}^+$ on $L^p_{W}(\mathbb{D}),$ where $W(\xi)=|\psi'(\xi)|^{2-p}|1-\overline{\xi}_{m}\xi|^t,$ $\xi \in \mathbb{D}$. 
	Let $f \in A_{\omega}^p(\Omega)$, $z\in \Omega$ and $v(w)=\text{dist}(w,\Omega)$.
	We have as in \eqref{finite}
	\begin{align}
	\nonumber
	|(T_{a,\Omega}f)(z)|&=\left|\sum_{n=1}^{\infty} \int_{S_n}K_{\Omega}(z,w)f(w)a(w)dA(w)\right|\leq\sum_{n=1}^{\infty}|F_nf(z)|
	\\
	\nonumber
	&\leq
	144C \int_{\Omega}|K_{\Omega}(z,w)|\Big(|f|+|f'|v(w)+|f''|v(w)^2\Big)dA(w)
	\\
	\label{finite2}
	&= C' P_{\Omega}^+\Big(|f|+|f'|v+|f''|v^2\Big)(z)< \infty,
	\end{align}
	where the integral is finite as was shown in \eqref{hölder} (note the comments after \eqref{hölder} regarding to the case $p\geq 4/3$).
	We obtain by changing variables $w=\psi(\xi)$, $\xi=\varphi(w)$ (Jacobian $|\psi'(\xi)|^2$),
	\begin{align}
	\left\Vert T_{a,\Omega}f\right\Vert_{p,\omega}
	\nonumber
	&\leq 
	C''
	\left\Vert P_{\Omega}^+\Big(|f|+|f'|v+|f''|v^2\Big)
	\right\Vert_{p,\omega}
	\\
	\label{maximaldisk1}
	&
	=C''\left\Vert P_{\mathbb{D}}^+\Big(|\psi'|\big(|f\circ \psi|+|f'\circ\psi|(v\circ\psi)+|f''\circ\psi|(v\circ \psi)^2\big)\Big)
	\right\Vert_{p,W},
	\end{align}
	where the weight $W(\xi)=|\psi'(\xi)|^{2-p}|1-\overline{\xi}_{m}\xi|^{t}$ and $\xi_m=\varphi(\boldsymbol{w_m})\in \partial\mathbb{D}$.
	For details about the above change of variables, see formula \eqref{change} in the last section.
	Now, recall the condition for the exponent $t$ and the expression of $\psi'(\xi)$ (see \eqref{schwarz1}). $P^+_{\mathbb{D}}$ is clearly bounded on $L^p_{W}(\mathbb{D})$ according to Corollary \ref{unit} (replace $\boldsymbol{\psi'}$ in Corollary \ref{unit} by $\psi'(\xi)|1-\overline{\xi}_{m}\xi|^{t/(2-p)}$ in the case  $p\geq 2$). Remark that $\left\Vert \left|\psi'\right|\left|f\circ \psi\right|\right\Vert_{p,W}=\left\Vert f \right\Vert_{p,\omega}<\infty$ since $f\in A_{\omega}^p(\Omega)$. The other functions in \eqref{maximaldisk1}, inside the brackets, also belong to $L_{W}^p(\mathbb{D})$ and their norms are bounded above by $\left\Vert f \right\Vert_{p,\omega}$. We show this only for  the norm of $\left|\psi'\right|\left|f'\circ\psi\right|(v\circ\psi)$, the norm of $\left|\psi''\right|\left|f''\circ\psi\right|(v\circ\psi)^2$ can be considered in the same way. By changing variables $w=\psi(\xi)$, we obtain (note that the following arguments are similar to those of the proof of Lemma \ref{corderivative}):
	\begin{align}	\nonumber\left\Vert\left|\psi'\right|\left|f'\circ\psi\right|(v\circ\psi)\right\Vert_{p,W}&=\left\Vert vf'\right\Vert_{p,\omega}=\left\Vert \omega^{1/p}vf'\right\Vert_{p}=\left\Vert \big(1-\varphi(w)\overline{\varphi(\boldsymbol{w_{m}})}\big)^{t/p}vf'\right\Vert_{p}\\&\nonumber=\left\Vert vD\big((1-\varphi(w)\overline{\varphi(\boldsymbol{w_{m}})}\big)^{t/p}f\big)+t/p\overline{\varphi(\boldsymbol{w_{m}})}\big(1-\overline{\varphi(\boldsymbol{w_{m}})}\varphi(w)\big)^{t/p-1}vf\varphi'\right\Vert_{p}\\&\nonumber\leq C\left\Vert v D\big(\omega^{1/p}f\big)\right\Vert_{p}+C\left\Vert \omega^{1/p-1/t}vf\varphi'\right\Vert_{p}\\&
	\label{vert}
	\leq C'\left\Vert \omega^{1/p}f\right\Vert_{p}=C'\left\Vert f\right\Vert_{p,\omega}<\infty,
	\end{align} 
	where the last line follows from Lemma \ref{corderivative} and from the fact that  $|(\omega(w))^{-1/t}v(w)\varphi'(w)|\leq C|(\omega(w))^{-1/t}(1-|\varphi(w)|^2)|\leq C'$ for all $w\in\Omega$ (recall \eqref{weight}).
	Thus we can continue from \eqref{maximaldisk1} 
	\begin{align*}
	\left\Vert T_{a,\Omega}f\right\Vert_{p,\omega}
	&
	\leq C'''\left\Vert|\psi'|\Big(|f\circ \psi|	+|f'\circ\psi|(v\circ\psi)+|f''\circ\psi|(v\circ \psi)^2\Big)
	\right\Vert_{p,W}
	\leq C''''\left\Vert f\right\Vert_{p,\omega},
	\end{align*}
	where the last inequality holds since each term is bounded by $\left\Vert f\right\Vert_{p,\omega}$. This proves the claim.
\end{proof}
\paragraph{Proof of Proposition \ref{main3}}
\begin{proof} 
 The key point of the proof is that the strong condition \eqref{symbol3} of the symbol guarantees the convergence of the integral \eqref{finite}.
	
	When bounding $|F_{k,n}f|$ (see \eqref{F1}--\eqref{F4}), we replace the condition $|\hat{a}_S(z')|\leq \mathbf{C}$ by condition $|\hat{a}_S(z')|\leq \mathbf{C}|G_{\Omega}(z',\boldsymbol{w_m})|^{t}$. Let $f \in A^p(\Omega)$ and $z\in \Omega$.
	Instead of \eqref{Fnestimate} we end up in the estimate
	\[
	|F_nf(z)|\leq C\int_{\tilde{S}_n}
\left|K_{\Omega}(z,w)\right|\left|G_{\Omega}(w,\boldsymbol{w_m})\right|^{t}\Big(|f(w)|+v(w)|f'(w)|+v(w)^2|f''(w)|\Big)dA(w).
	\]
	Thus we obtain (compare this to \eqref{finite})
	\begin{align}
	\nonumber
	|(T_{a,\Omega}f)(z)|&=\left|\sum_{n=1}^{\infty} \int_{S_n}K_{\Omega}(z,w)f(w)a(w)dA(w)\right|\leq\sum_{n=1}^{\infty}|F_nf(z)|
	\\
	\nonumber
	&\leq
	144C \int_{\Omega}\left|K_{\Omega}(z,w)\right|\left|G_{\Omega}(w,\boldsymbol{w_m})\right|^{t}\big(|f|+|f'|v+|f''|v^2\big)dA(w)
	\\
	\label{finite3}
	&= C' P_{\Omega}^+\Big(\left|G_{\Omega}(w,\boldsymbol{w_m})\right|^{t}\big(\left|f\right|+\left|f'\right|v+\left|f''\right|v^2\big)\Big)(z)< \infty,
	\end{align}
	where the integral is finite by Hölder's inequality as in \eqref{hölder}. The factor $|G_{\Omega}(w,\boldsymbol{w_m})|^{t}$, with the assumption for $t$, ensures now the convergence of the integral. (See the comments after \eqref{hölder} regarding to the case $1<p<4/3$.)
	Changing variables $\xi=\varphi(w)$ ($\xi_m=\varphi(\boldsymbol{w_m})$) yields
	\begin{align}
	\left\Vert T_{a,\Omega}f\right\Vert_{p}
	\nonumber
	&\leq 
	C''
	\left\Vert P_{\Omega}^+\Big(|G_{\Omega}(w,\boldsymbol{w_m})|^{t}(|f|+|f'|v+|f''|v^2)\Big)
	\right\Vert_{p}
	\\ \nonumber
	&
	=C''\left\Vert P_{\mathbb{D}}^+\Big(|1-\overline{\xi}_{m}\xi|^{t}|\psi'|\big(|f\circ \psi|+|f'\circ\psi|(v\circ\psi)+|f''\circ\psi|(v\circ \psi)^2\big)\Big)
	\right\Vert_{p,|\psi'|^{2-p}}
	\\
	\label{maximaldisk}
	&
	\leq C'''\left\Vert P_{\mathbb{D}}^+\Big(|1-\overline{\xi}_{m}\xi|^{t}|\psi'|\big(|f\circ \psi|+|f'\circ\psi|(v\circ\psi)+|f''\circ\psi|(v\circ \psi)^2\big)\Big)
		\right\Vert_{p,W},
	\end{align}
	where $W(\xi)=|\psi'(\xi)|^{2-p}|1-\overline{\xi}_{m}\xi|^{-t}$. We have added the factor $|1-\overline{\xi}_{m}\xi|^{-t},$ $t>0,$ to the weight $|\psi'(\xi)|^{2-p}$ in the last inequality.
	Now, recall the condition for the exponent $t$ and \eqref{schwarz1}. The operator $P^+_{\mathbb{D}}$ is bounded on $L^p_{W}(\mathbb{D})$ according to Corollary \ref{unit} (replace $\boldsymbol{\psi'}$ in Corollary \ref{unit} by $\psi'(\xi)|1-\overline{\xi}_{m}\xi|^{-t/(2-p)}$ in the case $p\leq 2$). Note also that $|1-\overline{\xi}_{m}\xi|^{t}\left|\psi'\right|\left|f\circ \psi\right|$ belongs to $L_{W}^p(\mathbb{D})$, as well as the other functions in \eqref{maximaldisk}, see Lemma \ref{corderivative}. Hence, we can continue from \eqref{maximaldisk} and make the change of variables $\xi=\varphi(w)$ once again 
	\begin{align*}
	\left\Vert T_{a,\Omega}f\right\Vert_{p}
	&
	\leq C''''\left\Vert\left|1-\overline{\xi}_{m}\xi\right|^{t}|\psi'|\Big(|f\circ \psi|	+|f'\circ\psi|(v\circ\psi)+|f''\circ\psi|(v\circ \psi)^2\Big)
	\right\Vert_{p,W}\\
	&
	= C''''\left\Vert\left|G_{\Omega}(w,\boldsymbol{w_m})\right|^{t}\big(\left|f\right|+\left|f'\right|v+\left|f''\right|v^2\big)
	\right\Vert_{p,\left|G_{\Omega}(w,\boldsymbol{w_m})\right|^{-t}}
	\\
	&
	= C''''\left\Vert\left|G_{\Omega}(w,\boldsymbol{w_m})\right|^{t(1-1/p)}\big(|f|+|f'|v+|f''|v^2\big)
	\right\Vert_{p}
	\leq C'''''\left\Vert f\right\Vert_{p},
	\end{align*}
	where $|G_{\Omega}(w,\boldsymbol{w_m})|^{t(1-1/p)}<C$ since $t(1-1/p)>0$. Thus, the last inequality holds by Lemma \ref{corderivative}. This completes the proof.
\end{proof}
\section{Examples}
In this section we deal with the classical Toeplitz operator $T_a$ (see \eqref{Tomega}) acting on the Bergman space $A^p(\Omega)$ where $\Omega$ is a polygon and $1<p<4/3$ or $p>4$. 
We consider the situation where $\Omega$ has such a large angle that Bergman projection is unbounded.
Despite the unboundedness of the Bergman projection it is possible that Toeplitz operator $T_a:A^p(\Omega)\rightarrow A^p(\Omega)$ is bounded, provided that strong enough conditions are set to the symbol.
It seems that finding good sufficient conditions for the boundedness is even more difficult than in the case of bounded Bergman projection; the question seems to be connected with two-weight inequalities for the Bergman projection, which are not well understood yet. The existing literature seems to contain only results for radial weights \cite{rattya}, which is not sufficient for our purposes. 
The basic problem is the singularity of the Bergman kernel. We will study the question by presenting some examples.

In the first example (Example \eqref{e0}) we consider the case $1<p<4/3$ and the rest deal with the case $p>4$. We assume that the vertex $\boldsymbol{\xi_{m}}\in \partial \Omega$ with the maximum angle $\boldsymbol{\alpha_m}\pi$ is attained at the point $1\in \partial \mathbb{D}$, i.e., $\psi(1)=\boldsymbol{\xi_{m}}$ and $\varphi(\boldsymbol{\xi_m})=1$. (We prefer to use $\xi$ instead of $w$ as a variable of $\Omega$ in this context.)	We need in all the examples the following change of variables formula: $w=\varphi(\xi)$, $z=\varphi(\lambda)$  (with Jacobian $|\psi'|^2$ and $\psi'(w)=(\varphi'(\xi))^{-1}$, $\psi'(z)=(\varphi'(\lambda))^{-1}$),
\begin{align}
\label{change}
\nonumber
\left\Vert T_a(f)\right\Vert^p_{p,\Omega}&=\int_{\Omega}\left|\int_{\Omega}\frac{\varphi'(\lambda)\overline{\varphi'(\xi)}f(\xi)a(\xi)}{\left(1-\varphi(\lambda)\overline{\varphi(\xi)}\right)^2}dA(\xi)\right|^p dA(\lambda)\\&=\int_{\mathbb{D}}|\psi'(z)|^{2-p}\left|\int_{\mathbb{D}}\frac{f(\psi(w))a(\psi(w))\psi'(w)}{(1-z\overline{w})^2}dA(w)\right|^p dA(z).
\end{align}
\begin{exmp}
	\label{e0}
	 Let $1<p<4/3$ and let 
	$\Omega\subset \mathbb{C}$ be a polygon with a large enough inward corner, so that Bergman projection $P_{\Omega}$ is unbounded (see Theorem \ref{Regulated}, case $p\leq 2$). More precisely, let $\boldsymbol{\alpha_m}=\max_k(\alpha_k)>1+\frac{2(p-1)}{2-p}$ and let $0<\alpha_k<1$ for all $k\neq \boldsymbol{m}$, i.e., except for the maximum angle the angles are outward. 
	\begin{enumerate}[$(a)$]
 \item
	Let symbol $a:\Omega\rightarrow \mathbb{C}$ be the constant function $a\equiv 1.$ We show that the Toeplitz operator $T_{a}$ is not even well-defined on $A^p(\Omega)$ (note the assumption for the Bergman kernel $K_{\Omega}(z,\cdot)$ in \cite[Proposition 2.4]{hedenmalm}). Let $f\in A^p(\Omega)$ and $\lambda \in \Omega$. We obtain by changing variables ($w=\varphi(\xi)$) \begin{equation}\label{lemme}(T_af)(\lambda)=(P_{\Omega}f)(\lambda)=\varphi'(\lambda)\int_{\mathbb{D}}\frac{f(\psi(w))\psi'(w)}{(1-\varphi(\lambda)\overline{w})^2}dA(w).\end{equation}
	In other words, $(T_a f)(\lambda) =\varphi'(\lambda)P_{\mathbb{D}}\big((f\circ\psi)\psi'\big)(\varphi(\lambda))$, where $(f\circ\psi)\psi'\in A^p_{|\psi'|^{2-p}}(\mathbb{D})$. 
	It is stated in \cite[Lemme 4]{bekolle1} that $P_{\mathbb{D}}$ is well-defined on $L^p_{\omega}(\mathbb{D})$ only if $\omega^{-1/(p-1)}$ is integrable. 
	Since we consider $P_{\mathbb{D}}$ on the Bergman space $A_{|\psi'|^{2-p}}^p(\mathbb{D})$ instead of $L_{|\psi'|^{2-p}}^p(\mathbb{D})$, we show in detail that the claim holds:

	 Let $f$ be $f(\xi):=(1-\overline{\varphi(\boldsymbol{\xi_m})}\varphi(\xi))^{-1-\boldsymbol{\alpha_m}}$, $\xi\in \Omega$, i.e., $f(\psi(w))=(1-w)^{-1-\boldsymbol{\alpha_m}}$, $w\in \mathbb{D}$. Remark that $f \in A^p(\Omega)$ by the assumption $\boldsymbol{\alpha_m}>1+\frac{2(p-1)}{2-p}$. 
	 Now, the integral in \eqref{lemme} diverges for all $\lambda\in \Omega$, since 	\[
	\int_{\mathbb{D}}\frac{|1-w|^{-1-\boldsymbol{\alpha_m}}|\psi'(w)|}{|1-\varphi(\lambda)\overline{w}|^2}dA(w)\geq C_{\lambda}\int_{B(1,c)\cap\mathbb{D}}|1-w|^{-2}dA(w)=\infty,
	\]
	where $B(1,c)$ is a disk which does not intersect with any of $w_k=\varphi(\xi_k)$, $k\neq \boldsymbol{m}$. We have replaced $|\psi'(w)|$ by $|1-w|^{\boldsymbol{\alpha_{m}}-1}$, see \eqref{schwarz1}.
	 \item We now modify the constant symbol of (a) such that it will induce a bounded Toeplitz operator $T_a:A^p(\Omega)\rightarrow A^p(\Omega)$. Let $a:\Omega\rightarrow \mathbb{C}$ be such that $|a(\xi)|\leq C|1-\overline{\varphi(\boldsymbol{\xi_m})}\varphi(\xi)|^t$ for all $\xi \in \Omega$ and $t> (2-p)(\boldsymbol{\alpha_m}-1)-2(p-1)>0$. Let $f\in A^p(\Omega)$. Changing variables as in \eqref{change} leads to
	 \begin{align*}
	 \left\Vert T_a(f)\right\Vert^p_{p,\Omega}&\leq C\int_{\mathbb{D}}|\psi'(z)|^{2-p}\left(\int_{\mathbb{D}}\frac{|f(\psi(w))|1-w|^t|\psi'(w)|}{|1-z\overline{w}|^2}dA(w)\right)^p dA(z)
	 \\&\leq C' \int_{\mathbb{D}}|\psi'(z)|^{2-p}|1-z|^{-t}\left(\int_{\mathbb{D}}\frac{|f(\psi(w))|1-w|^t|\psi'(w)|}{|1-z\overline{w}|^2}dA(w)\right)^p dA(z)
	  \\&\leq C'' \int_{\mathbb{D}}|\psi'(z)|^{2}|1-z|^{t(p-1)}|f(\psi(z))|^p
	  \leq C'''\left\Vert f\right\Vert^p_{p,\Omega},
	 \end{align*}
	 where the inner integral converges because of the factor $|1-w|^t$ (this is easy to check by applying Hölder's inequality). 
	 On the second line we have applied the boundedness of 
	 $P_{\mathbb{D}}^+$ on $L_W^p(\mathbb{D})$, where $W(z)=|\psi'(z)|^{2-p}|1-z|^{-t}$ (see Corollary \ref{unit} case $p\leq 2$). 
		\end{enumerate} 
\end{exmp}
In the next example (part (b)) we remark that even though the symbol has a compact support, it is possible that the Toeplitz operator acting on $A^p(\Omega)$ is not even mapping into $A^p(\Omega)$.
\begin{exmp}
	\label{e1}
	Let $p>4$ and let 
	$\Omega\subset \mathbb{C}$ be a polygon with a large enough (inward) corner, so that Bergman projection is unbounded. More precisely, let $\boldsymbol{\alpha_m}=\max_k(\alpha_k)>1+\frac{2}{p-2}$ and let $0<\alpha_k<1$ for all $k\neq \boldsymbol{m}$, i.e., except for the maximum angle the angles are outward.
	\begin{enumerate}[$(a)$]  
\item
We define a bounded symbol $a:\Omega\rightarrow \mathbb{C}$, $a(\xi):=\varphi'(\xi)(1-|\varphi(\xi)|^2)$. 
We claim that $T_a$ operating on $A^p(\Omega)$ is not even mapping into $A^p(\Omega)$.
Indeed, if we take $f\in A^p(\Omega)$ such that $f\equiv 1$ and change variables as in \eqref{e0}, we have
\begin{eqnarray}
	\nonumber\left\Vert T_a(f)\right\Vert^p_{p,\Omega}
	&=&\int_{\mathbb{D}}|\psi'(z)|^{2-p}\left|\int_{\mathbb{D}}\frac{1-|w|^2}{\left(1-z\overline{w}\right)^2}dA(w)\right|^p dA(z)
	\\\nonumber &\geq& \frac{C}{2^p}\int_{B(1,c)\cap\mathbb{D}}|1-z|^{(\boldsymbol{\alpha_{m}}-1)(2-p)} dA(z)
	=\infty,
\end{eqnarray}
where we have applied the mapping property \eqref{mapa} to the inner integral with variable $w$ and $B(1,c)$ is as defined in Example \ref{e0} (a).
	\item Now, let symbol $a:\Omega\rightarrow \mathbb{C}$ be 
	$a(\xi):=\varphi'(\xi)(\mbox{\Large$\chi$}_{B(0,R)}\circ \varphi)(\xi)$ for all $\xi\in\Omega$,
	where $0<R<1$ is the radius of the disk $B(0,R)$ and $\chi_{B(0,R)}$ is the characteristic function. Obviously there is $r>0$ such that $a(\xi)=0$ when $\text{dist}(\xi,\partial \Omega)<r$. However, $T_a$ operating on $A^p(\Omega)$ is not even mapping into $A^p(\Omega)$: Let $f\in A^p(\Omega)$ be the constant function $f\equiv 1$. 
	By changing variables $w=\varphi(\xi)$ and $z=\varphi(\lambda)$ as in \eqref{change}, we get 
	\begin{align}
	\nonumber\left\Vert T_a(f)\right\Vert^p_{p,\Omega}&= \int_{\mathbb{D}}|\psi'(z)|^{2-p}\left|\int_{\mathbb{D}}\frac{\mbox{\Large$\chi$}_{B(0,R)}(w)}{\left(1-z\overline{w}\right)^2}dA(w)\right|^p dA(z)
	\\\nonumber
	&= C_R\int_{\mathbb{D}}|\psi'(z)|^{2-p}dA(z)=\infty,
	\end{align}
	where the last integral diverges by \eqref{schwarz1} and by condition $\boldsymbol{\alpha_{m}}>1+\frac{2}{p-2}$. 
\end{enumerate}
\end{exmp}
However, some special symbols induce bounded Toeplitz operators.
In the next two examples we define bounded Toeplitz operators by modifying the symbol of Example \ref{e1} (a). 
This modification is not trivial, i.e., it is not trivial to find anti-analytic non-zero $L^1(\Omega)$-symbol that defines a bounded Toeplitz operator on $A^p(\Omega)$ in the case $p>4$.
\begin{exmp}
	Let $p>4$ and assume that $\Omega$ is as in Example \ref{e1}. Let us define (compare this to the symbol of Example \ref{e1} (a))
	\[
	a(\xi):=(\varphi'(\xi))^{1-2/p}(1-|\varphi(\xi)|^2)\left(1+|\varphi(\xi)|^2-2\varphi(\xi)\right),\;\;\;\xi\in \Omega,
	\]
	Thus ($w=\varphi(\xi)$) 	\[a(\psi(w))=(\psi'(w))^{2/p-1}(1-|w|^2)\left(1+|w|^2-2w\right), \;\;\;w\in \mathbb{D}.\]
	We claim that $T_a:A^p(\Omega)\rightarrow A^p(\Omega)$ is bounded.
	An easy computation gives (see Lemma \ref{bergman_prop})
	\[P_{\mathbb{D}}\left((1-|z|^2)(1+|z|^2-2z)z^n\right)=\frac{2z^n(1-z)}{n+3}\;\;\;\text{for all}\;\;\;n=0,1,...\]
	Let $f\in A^p(\Omega)$ and let $\sum_{n=0}^{\infty}a_nw^n$ be the Taylor series of the analytic function $(f\circ\psi)(\psi')^{2/p}$. Recall \eqref{schwarz1} and remark that $\alpha_k<1$ for all $k\neq \boldsymbol{m}.$ Hence $|\psi'(z)|\geq C|1-z|^{\boldsymbol{\alpha_{m}}-1}$ for all $z\in \mathbb{D}$. By changing variables $\xi=\psi(w)$ and $\lambda=\psi(z)$ and substituting $a(\psi(w))=(\psi'(w))^{2/p-1}(1-|w|^2)(1+|w|^2-2w)$ into \eqref{change} we get
	\begin{align*}
		\left\Vert T_a(f)\right\Vert^p_{p,\Omega}&=\int_{\mathbb{D}}|\psi'(z)|^{2-p}\left|\int_{\mathbb{D}}\frac{f(\psi(w))a(\psi(w))\psi'(w)}{(1-z\overline{w})^2}dA(w)\right|^p dA(z)
		\\
		&\leq C\int_{\mathbb{D}}|1-z|^{(\boldsymbol{\alpha_m}-1)(2-p)}\left|\int_{\mathbb{D}}\frac{f(\psi(w))(\psi'(w))^{2/p}(1-|w|^2)(1+|w|^2-2w)}{(1-z\overline{w})^2}dA(w)\right|^p dA(z)\\
		&= C\int_{\mathbb{D}}|1-z|^{(\boldsymbol{\alpha_m}-1)(2-p)}\left|\int_{\mathbb{D}}\frac{(\sum_{n=0}^{\infty}a_nw^n)(1-|w|^2)(1+|w|^2-2w)}{(1-z\overline{w})^2}dA(w)\right|^p dA(z)\\
		&=C'\int_{\mathbb{D}}|1-z|^{(\boldsymbol{\alpha_m}-1)(2-p)}\left|(1-z)\sum_{n=0}^{\infty}\frac{a_n}{n+3}z^n\right|^p dA(z)\\
		&=C'\int_{\mathbb{D}}|1-z|^{(\boldsymbol{\alpha_m}-1)(2-p)+p}\left|\sum_{n=0}^{\infty}\frac{a_n}{n+3}z^n\right|^p dA(z)\\
		&\leq C''\int_{\mathbb{D}}\left|\sum_{n=0}^{\infty}\frac{a_n}{n+3}z^n\right|^p dA(z) \\
		&\leq C'''\int_{\mathbb{D}}|f\circ \psi|^p|\psi'|^2dA=C'''\left\Vert f\right\Vert^p_{p,\Omega},
	\end{align*}
	where we have taken into account that $p>4$ and $\boldsymbol{\alpha_m}<2$, so that $(\boldsymbol{\alpha_m}-1)(2-p)+p>0$. The last inequality follows from \eqref{taylorcof}. This shows that $T_a$ is bounded.
\end{exmp}
\begin{exmp}
\label{e3}
Let $p>4$ and assume that $\Omega$ is as in Example \ref{e1}. Let us define (compare this to the symbol of Example \ref{e1} (a))
\[
a(\xi):=(\varphi'(\xi))^{1-2/p}(1-|\varphi(\xi)|^2)\left(\frac{\varphi(\xi)}{|\varphi(\xi)|}-|\varphi(\xi)|\right)^m,\;\;\;\xi\in \Omega,
\]
where $m\in\{2,3,...\}$. 
Hence in the unit disk ($w=\varphi(\xi)$)
\[
a(\psi(w))=(\psi'(w))^{2/p-1}(1-|w|^2)\left(\frac{w}{|w|}-|w|\right)^m,\;\;\;w\in \mathbb{D}.
\]
We claim that $T_{a}:A^p(\Omega)\rightarrow A^p(\Omega)$ is bounded. 
	Let $f \in A^p(\Omega)$.
We write $g:=(f\circ\psi)(\psi')^{2/p}$, so that $\left\Vert f\right\Vert_{p,\Omega}=\left\Vert g\right\Vert_{p,\mathbb{D}}$. Moreover, let $g(w)=\sum_{n=0}^{\infty}a_nw^n$, where $\sum_{n=0}^{\infty}a_nw^n$ is the Taylor series of $g.$ We make the change of variables $w=re^{i\theta}=\varphi(\xi)$ and $z=\varphi(\lambda)$ as in \eqref{change}:
\begin{align}
\label{e3estimate}
\nonumber\left\Vert T_a(f)\right\Vert^p_{p,\Omega}
&=\nonumber \int_{\mathbb{D}}|\psi'(z)|^{2-p}\left|\int_{\mathbb{D}}\frac{(f(\psi(w))(\psi'(w))^{2/p}(1-|w|^2)(w/|w|-|w|)^m}{\left(1-z\overline{w}\right)^2}dA(w)\right|^p dA(z)
\\\nonumber&\leq C\int_{\mathbb{D}}|1-z|^{(\boldsymbol{\alpha_{m}}-1)(2-p)}\left|\int_{0}^{2\pi}\int_{0}^{1}\frac{\big(\sum_{n=0}^{\infty}a_nr^ne^{in\theta}\big)(1-r^2)(e^{i\theta}-r)^mrdrd\theta/\pi}{\left(1-zre^{-i\theta}\right)^2}\right|^p dA(z)
\\&\leq C'\int_{\mathbb{D}}|1-z|^{(\boldsymbol{\alpha_{m}}-1)(2-p)}\left|\sum_{n=0}^{\infty}a_n\int_{0}^{1}(1-r^2)r^{n+1}\left(\int_{0}^{2\pi}\frac{e^{in\theta}(e^{i\theta}-r)^md\theta}{\left(1-zre^{-i\theta}\right)^2}\right)dr\right|^p dA(z).
\end{align}
Let us now take a closer look at the innermost integral with integration variable $\theta$. Integration by parts yields
\begin{align}
\nonumber\int_{0}^{2\pi}\frac{e^{in\theta}(e^{i\theta}-r)^md\theta}{\left(1-zre^{-i\theta}\right)^2}&=\int_{0}^{2\pi}\frac{D_{\theta}(e^{i(n+1)\theta}(e^{i\theta}-r)^m)d\theta}{zri(1-zre^{-i\theta})}
\\
\nonumber&=\int_{0}^{2\pi}\frac{(n+1)e^{i(n+1)\theta}(e^{i\theta}-r)^m+me^{i(n+2)\theta}(e^{i\theta}-r)^{m-1}}{zr(1-zre^{-i\theta})}d\theta
\\
\nonumber&=\frac{2\pi}{zr}(Sh_r)(zr)=\frac{2\pi}{zr}h_r(zr)
\\
\nonumber&=\frac{2\pi}{zr}\Big((n+1)(zr)^{n+1}(zr-r)^m+m(zr)^{n+2}(zr-r)^{m-1}\Big)
\\\nonumber
&=2\pi(z-1)^{m-1}\big((n+m+1)z^{n+1}-(n+1)z^n\big)r^{n+m},
\end{align}
where $S$ is the Cauchy-Szegö integral operator \eqref{szego} and $h_r:\overline{\mathbb{D}}\rightarrow \mathbb{C}$, $h_r(w):=(n+1)w^{n+1}(w-r)^m+mw^{n+2}(w-r)^{m-1}$ (see the numerator of the integrand on the second line). Note that $h_r$ is analytic on $\mathbb{D}$ and continuous on $\overline{\mathbb{D}}$ since $m$ is an integer $\geq 2$. Now, the factor $(z-1)^{m-1}$ on the last line is essential for the boundedness of the Toeplitz operator. We can continue from \eqref{e3estimate}
\begin{align}
\nonumber &\left\Vert T_a(f)\right\Vert^p_{p,\Omega}
\leq
C'\int_{\mathbb{D}}|1-z|^{(\boldsymbol{\alpha_{m}}-1)(2-p)+(m-1)p}\left|\sum_{n=0}^{\infty}a_n\big((n+m+1)z^{n+1}\right.\\
&\quad\quad\quad\quad\quad\quad\left.-(n+1)z^n\big)\int_{0}^{1}(1-r^2)r^{2n+m+1}dr\right|^p dA(z)\nonumber
\\&\quad\quad
\nonumber
\leq C''\int_{\mathbb{D}}\left|\sum_{n=0}^{\infty}a_n\big((n+m+1)z^{n+1}-(n+1)z^n\big)\int_{0}^{1}(1-r^2)r^{2n+m+1}dr\right|^p dA(z)\\
\nonumber&\quad\quad= C''\int_{\mathbb{D}}\left|\sum_{n=0}^{\infty}\frac{2(n+m+1)a_nz^{n+1}}{(2n+m+2)(2n+m+4)}-\sum_{n=0}^{\infty}\frac{2(n+1)a_nz^{n}}{(2n+m+2)(2n+m+4)}\right|^p dA(z)\\\nonumber
&\quad\quad\leq C'''\int_{\mathbb{D}}\left|\sum_{n=0}^{\infty}\frac{(n+m+1)a_nz^n}{(2n+m+2)(2n+m+4)}\right|^p dA(z)+C'''\int_{\mathbb{D}} \left|\sum_{n=0}^{\infty}\frac{(n+1)a_nz^{n}}{(2n+m+2)(2n+m+4)}\right|^p dA(z)\\
\nonumber
&\quad\quad\leq C''''\left\Vert g\right\Vert^p_{p,\mathbb{D}}=C''''\left\Vert f\right\Vert^p_{p,\Omega},
\end{align}
where we have applied the estimate $\left\Vert h_1+h_2\right\Vert^p_p\leq 2^p\text{max}(\left\Vert h_1\right\Vert^p_p,\left\Vert h_2\right\Vert^p_p)$ and on the second last line the relation \eqref{taylorcof}. To make sure that \eqref{taylorcof} works, write 
$h_1(z)=\sum_{n=0}^{\infty}b_nz^n$ with coefficients $b_n=(n+m+1)a_n/\big((2n+m+2)(2n+m+4)\big)$ and $h_2(z)=\sum_{n=0}^{\infty}c_nz^n$ with $c_n=(n+1)a_n/\big((2n+m+2)(2n+m+4)\big)$ and recall that $g(z)=\sum_{n=0}^{\infty}a_nz^n.$ We have shown that $T_a$ is bounded.
\end{exmp}
\section*{Acknowledgement}
The author would like to thank her supervisor Jari Taskinen (University of Helsinki) for his advice and suggestions.
\bibliography{citation_polygonal}{}
\bibliographystyle{acm}
\medskip
\medskip
Paula Mannersalo\\
Department of Mathematics and Statistics\\
P.O. Box 68\\
FI-00014 University of Helsinki\\
Finland\\
E-mail: \url{paula.mannersalo@gmail.com}
\end{document}